\documentclass{daj}

\usepackage{amsthm}
\usepackage{amsmath}
\usepackage{amssymb}

\newcommand{\PP}{\mathbb{P}}
\newcommand{\AAA}{\mathbb{A}}
\newcommand{\ZZ}{\mathbb{Z}}
\newcommand{\NN}{\mathbb{N}}
\newcommand{\QQ}{\mathbb{Q}}

\newcommand{\RR}{\mathbb{R}}
\newcommand{\F}{\mathbb{F}}
\newcommand{\cE}{\mathcal{E}}
\newcommand{\M}{\mathbf{M}}
\newcommand{\ve}{\varepsilon}
\newcommand{\uu}{\underline}
\newcommand{\abs}[1]{\left\lvert#1\right\rvert}
\DeclareMathOperator{\rank}{rank}
\DeclareMathOperator{\Pic}{Pic}

\newtheorem{theorem}{Theorem}[section]
\newtheorem{corollary}[theorem]{Corollary}
\newtheorem{lemma}[theorem]{Lemma}

\theoremstyle{definition}

\newtheorem{remark}[theorem]{Remark}

\dajAUTHORdetails{%
	title = {Counting integer points on affine surfaces with a side condition}, 
	author = {Tim Browning and Matteo Verzobio},
	plaintextauthor = {Tim Browning, Matteo Verzobio},
	plaintexttitle = {Counting integer points on affine surface with a side condition}, 
}   

\dajEDITORdetails{%
	year={2025},
	number={12},
	received={22 August 2024},   
	revised={25 February 2025},    
	published={5 September 2025},  
	doi={10.19086/da.143787},       
}   

\begin{document}
	
	\begin{frontmatter}[classification=text]
		
		\title{Counting integer points on affine surfaces with a side condition} 
		
		\author[tb]{Tim Browning \thanks{Supported by FWF grant (DOI 10.55776/P36278)}}
		\author[mv]{Matteo Verzobio\thanks{Supported by European Union's Horizon 2020 research and 
				innovation program under the Marie Sk\l odowska-Curie Grant Agreement No. 
				101034413.}}
		
		\begin{abstract}
			We extend  work of Heath-Brown and Salberger, based on 
			the determinant method, 
			to
			provide a uniform upper bound for the number of integral points of bounded height 
			on an affine surface, which are subject to a polynomial congruence condition. 
			This is applied to get a new uniform bound for  points on diagonal quadric surfaces, and to a problem about the representation of integers as a sum of four unlike powers.
		\end{abstract}
	\end{frontmatter}
	
	
		\section{Introduction}

	Let $f\in\ZZ[x_1,x_2,x_3]$  be an absolutely irreducible polynomial
	of degree $d$. 
	The determinant method of Bombieri--Pila \cite{BP} has proved a tremendously useful tool for bounding the number $N(f;B)$ of zeros  $\uu{x}=(x_1,x_2,x_3)\in \ZZ^3$ of $f$ with $|x_1|,|x_2|,|x_3|\leq B$, for a large parameter $B$. 
	As noticed by Pila \cite{pila}, it is possible to apply the bounds in \cite{BP}  to prove that 
	$N(f;B)=O_{d,\ve}(B^{1+1/d+\ve})$ for any $\ve>0$, where the implied constant only depends on $d$ and $\ve$. The determinant method has been extensively developed, most notably by Heath-Brown \cite{annal,cime} and by Salberger \cite{salberger}. It follows from Salberger's work \cite[Cor.~7.3]{salberger} that 
	\begin{equation}\label{eq:s1}
		N(f;B)=O_{d,\ve}(B^{2/\sqrt{d}+\ve}+B^{1+\ve}),
	\end{equation}
	which is now the state of the art for this problem. 
	
	In this paper we shall study a variant of this counting problem in which 
	lopsided boxes are allowed, and moreover, in which an additional side condition is imposed, in the form of a congruence.  Let $g\in \ZZ[x_1,x_2,x_3]$ be a further non-constant polynomial and let $q\in \NN$. Under suitable hypotheses we shall investigate the size of the set 
	$$
	S_q(f,g;\uu{B})=\left\{\uu{x}\in \ZZ^3: |x_i|\leq B_i,~ f(\uu{x})=0,~g(\uu{x})\equiv 0 \bmod{q}\right\},
	$$
	where  $\uu{B}=(B_1,B_2,B_3)$, with  $B_1,B_2,B_3\geq 2$.
	We simply write 
	$S_q(f,g;B)$, when  $B_1=B_2=B_3=B$.
	
	We will proceed under the  assumption that 
	$g$ does not depend on $x_1$, so that 
	$g\in \ZZ[x_2,x_3]$. When $g$ is homogeneous, one can cover the congruence condition 
	$g(\uu{x})\equiv 0 \bmod{q}$ by a small number of lattices. Techniques from the geometry of numbers then  reduce the problem to bounding 
	$\#S_1(\tilde f,0;\uu{C})$ for a suitable degree $d$ polynomial $\tilde f$ and a new triple 
	$ \uu{C}=( C_1,C_2,C_3)$  that depends on $q$.   The main aim of this paper is to harness 
	arguments of  Heath-Brown \cite{annal,cime} and Salberger \cite{salberger} to extract information about the case that 
	$g$ is not necessarily homogeneous. We shall apply our work to get new uniform bounds for the density of integral points on the quadric surface 
	$a_1x_1^2+a_2x_2^2+a_3x_3^2=n$, and on the threefold
	$x_1^k+x_2^{l}+x_3^{m}+x_4^k=n$.

	For given $f\in \ZZ[x_1,x_2,x_3]$ and $g\in \ZZ[x_2,x_3]$ we will need to introduce some notation 
	and conventions that will stay in place throughout this paper.  
	We will allow all implied constants to depend on the degrees of $f$ and $g$.
	Let $\uu{B}\in \RR^3$ with $B_1,B_2,B_3\geq 2$.
	We fix a strict total order $\prec$ on $\ZZ_{\geq 0}^3$. We shall assume that it is  linear, so that  $\uu{a}\prec \uu{b}$ and $\uu{c}\prec \uu{d}$ implies that 
	$\uu{a}+\uu{c}\prec  \uu{b}+\uu{d}$, for any $\uu{a},\uu{b},\uu{c},\uu{d}\in \ZZ_{\geq 0}^3$.
	We call any such order a linear total order. 
	The lexicographic order is an example of a linear total order (and it is this order that we shall use in 
	our applications).
	Among the monomials $\uu{x}^{\uu{m}}=x_1^{m_1}x_2^{m_2}x_3^{m_3}$ that appear in $f$ with a non-zero coefficient, we choose $\uu{m}\in \ZZ_{\geq 0}^3$ to be the  maximum for the order. We are going to work under the assumption that $\uu{m}\neq \uu{0}$. For this choice of $\uu{m}$, we put 
	\begin{equation}\label{def:T}
		T_{\uu{m}}=\uu{B}^{\uu{m}}=B_1^{m_1}B_2^{m_2}B_3^{m_3}.
	\end{equation} 
	Let 
	\begin{equation}\label{def:S}
		S=\max_{\uu{e}}\{\log \uu{B}^{\uu{e}}\},
	\end{equation}
	where the maximum runs over exponents $\uu{e}\in \ZZ_{\geq 0}^3$ such that  $\uu{x}^{\uu{e}}$ appears in $g$ with non-zero coefficient. (Note that $e_1=0$ in any such exponent vector by our assumption on $g$.)
	Let 
	\begin{equation}\label{eq:min-max}
		B'=\min\{B_1,B_2,B_3\} , \quad
		B=\max\{B_1,B_2,B_3\}, 
	\end{equation}
	and 
	\begin{equation}\label{eq:defK'}
		K=\exp\left(\sqrt\frac{\log B_1\log B_2\log B_3}{\log T_{\uu{m}}}\left(1-\frac{m_1\log B_1}{2S\log T_{\uu{m}}}\log q\right)\right).
	\end{equation} 
	Note that $\log T_{\uu{m}}$ and $S$ are not zero (since $\uu{m}\neq \uu{0}$ and $g$ is non-constant) and so $K$ is well-defined.
	Then, with this notation, the following 
	will be proved in Section~\ref{s:main}. 
	
	\begin{theorem}\label{thm:side}
		Let $\ve>0$ and $q\in \NN$.
		Let $f\in \ZZ[x_1,x_2,x_3]$ be absolutely irreducible and let  $g\in \ZZ[x_2,x_3]$.
		Assume that $\gcd(q,g(0,0))=1$ and 
		let $B,B',K$ be given by \eqref{eq:min-max} and \eqref{eq:defK'}.
		Then there exist 
		$f_1,\ldots, f_J\in \ZZ[x_1,x_2,x_3],$ 
		and a finite collection of points $Z\subset V(f)$, such that the following
		hold: 
		\begin{enumerate}
			\item $J=O_{\ve}(\max\{1,KB^\ve\})$;
			\item
			each $f_j$ is coprime to $f$ and has degree 
			$O_{\ve}(\log B/\log B')$, for $j\leq J$; 
			\item
			$\#Z=O_{\ve}(\max\{1,KB^\ve\}^{2})$; and 
			\item
			for each $\uu{x}\in S_q(f,g;\uu{B})\setminus Z$, 
			there exists $j\leq J$ such that
			$$
			f(x_1,x_2,x_3)=f_j(x_1,x_2,x_3)=0.
			$$
		\end{enumerate}
	\end{theorem}
	
	We emphasise that the implied constants in this result only  depend on the choice of $\ve$ and on the degrees of $f$ and $g$, a convention that will remain in place throughout this paper. 
	The proof  relies crucially on the breakthrough work of Salberger \cite{salberger} and is based on the strategy he used to prove \cite[Lemma~3.2]{salberger}.
	In order to find the best possible bound, one chooses the linear total order that minimises $K$.
	When  $q=1$, we may assume without loss of generality that the values of $\log B_i$ are linearly independent over $\QQ$. One may then check that 
	$\uu{a}\prec \uu{b}$ if and only if $\uu{B}^{\uu{a}}< \uu{B}^{\uu{b}}$ is a linear total order on $\ZZ_{\geq 0}^3$.
	Thus we  may take 
	$$		K=\exp\left(\sqrt{\frac{\log B_1\log B_2\log B_3}{\max_{\uu{m}}T_{\uu{m}}}}\right)
	$$
	in \eqref{eq:defK'},
	where the maximum is taken over 
	all exponents $\uu{m}\in \ZZ_{\geq 0}^3$ such that  $\uu{x}^{\uu{m}}$ appears in $f$ with non-zero coefficient. Thus 
	Theorem \ref{thm:side} recovers work of the first author \cite[Lemma~1]{tb}, in which  a similar result without the side condition was already extracted from  \cite{cime} and \cite{salberger}.
	(In fact, Theorem \ref{thm:side} corrects a mistake in \cite[Lemma~1]{tb}, in which it is claimed that  $f_j$ has degree $O_\ve(1)$ for all $j\leq J$.) This in turn could be used to recover the bound 
	\eqref{eq:s1} of Salberger when $B_1=B_2=B_3=B$.

	In the case $B_1=B_2=B_3=B$ we can relax the condition 
	$\gcd(q,g(0,0))=1$ somewhat. 
	In this case  \eqref{eq:defK'} becomes
	\begin{equation}\label{eq:B1B'}
		K=B^{\frac{1}{\sqrt{m_1+m_2+m_3}}}q^{-\frac{m_1}{2(m_1+m_2+m_3)^{3/2}\deg g}},
	\end{equation}
	for any $\uu{m}\in \ZZ_{\geq 0}^3$  that is a maximum (among the exponents of $f$) for a fixed linear total order. We shall also prove the following result
	in Section \ref{s:main}.

	\begin{theorem}\label{thm:side'}
		Let $\ve>0$ and $q\in \NN$. 
		Let $f\in \ZZ[x_1,x_2,x_3]$ be absolutely irreducible and let  $g\in \ZZ[x_2,x_3]$.
		Let 
		$K$ be given by  \eqref{eq:B1B'} and 
		assume that 
		\begin{equation}\label{eq:gcdq}
			\gcd(q,g(0,0),g_0(1,0),g_0(0,1))=1,
		\end{equation}
		where $g_0$ is the top degree part of $g$. 
		Then there exist 
		$f_1,\ldots, f_J\in \ZZ[x_1,x_2,x_3],$ 
		and a finite collection of points $Z\subset V(f)$, such that the following
		hold: 
		\begin{enumerate}
			\item $J=O_{\ve}(\max\{1,KB^\ve\})$;
			\item
			each $f_j$ is coprime to $f$ and has degree 
			$O_{\ve}(1)$, for $j\leq J$; 
			\item
			$\#Z=O_{\ve}(\max\{1,KB^\ve\}^{2})$; and 
			\item
			for each $\uu{x}\in S_q(f,g;B)\setminus Z$, 
			there exists $j\leq J$ such that
			$$
			f(x_1,x_2,x_3)=f_j(x_1,x_2,x_3)=0.
			$$
		\end{enumerate}
	\end{theorem}

	Salberger's version of the determinant method has been further refined to  capture the dependence of the error term on the degree $d$ of the polynomial. Thus it follows from work of 
	Binyamin et al \cite{bin} that a quadratic dependence is possible. 
	While we believe that this ought to give an alternative route to our main results, it would require a reworking of the proof and so we choose not to  pursue this direction here.
	
	\medskip
	
	We now collect two applications of our bound for surfaces with a side condition. The first concerns a uniform upper bound for the
	affine quadric surface $S\subset \AAA^3$ defined by $a_1x_1^2+a_2x_2^2+a_3x_3^2=n$,
	for suitable  coefficients $a_1,a_2,a_3,n\in \ZZ$.
	Let
	\begin{equation}\label{eq:N1}
		N_1(B)=\#\{\uu{x}\in S(\ZZ): |\uu{x}|\leq B\},
	\end{equation}
	where
	$	|\uu{x}|=\max\{|x_1|,|x_2|,|x_3|\}$ is the sup-norm. 
	In some applications it is useful to have a bound for $N_1(B)$ that gets better when $|\uu{a}|$ gets larger. When $n=0$ and the integer solutions are further restricted to be primitive, such a bound is provided in \cite[Cor.~2]{n-2} using arguments from the geometry of numbers. 
	We are interested here in the case $n\neq 0$.  In this case a divisor function bound (in the form \cite[Thm.~1.11]{bg}, for example)  readily yields
	\begin{equation}\label{eq:trivial}
		N_1(B)\ll_\ve B^{1+\ve},
	\end{equation}
	for any $\ve>0$, where the implied constant only depends on $\ve>0$. This bound is optimal up to 
	the factor $B^\ve$,
	since $S$ can contain rational lines
	if  $-a_1a_2a_3n= \square$.
	We are interested in obtaining bounds for $N_1(B)$ that improve on this when 
	$|\uu{a}|$ is  large and 
	$-a_1a_2a_3n\neq \square$.
	We shall build on the proof of Theorem~\ref{thm:side'} to prove the following estimate in Section \ref{s:quadric}, which improves on the divisor  bound as soon as  $|\uu{a}|> B^{1+\delta}$, for any fixed $\delta>0$.

	\begin{corollary}\label{thm:squares}
		Let $a_1,a_2,a_3,n\in\ZZ$, with $-a_1a_2a_3n\neq \square$
		and $\gcd(a_1,a_2,a_3,n)=1$. Then, for all $\ve>0$, we have 
		$$
		N_1(B)\ll_\ve\frac{B^{\frac{7}{6}+\ve}}{|\uu{a}|^{1/6}}+B^{\frac{1}{2}+\ve}.
		$$
	\end{corollary}
	
	A  direct application of Theorem \ref{thm:side'} would actually lead to the weaker bound 
	$$
	N_1(B)\ll_\ve 
	|\uu{a}|^{-1/(2\sqrt{2})}
	B^{\sqrt{2}+\ve}
	+
	|\uu{a}|^{-1/(4\sqrt{2})}
	B^{\frac{1}{\sqrt{2}}+\frac{1}{2}+\ve}
	+B^{\frac{1}{2}+\ve}.
	$$
	In order to arrive at the statement of 
	Corollary \ref{thm:squares}, we shall exploit earlier work of Salberger \cite{salb-cubic} that is  more efficient for surfaces of degree $2$ or $3$. 
	Note that a bound for $N_1(B)$ can also be extracted from work of Ellenberg and Venkatesh \cite[Prop.~1 and Rem.~2]{EV}, which uses the determinant method to get better upper bounds when the underlying polynomial has large coefficients. When specialised to our setting (with $d=2$ and $n=2$), their work 
	shows that the set of 
	$\uu{x}\in S(\ZZ)$ with $ |\uu{x}|\leq B$ are covered by $O_\ve(
	|\uu{a}|^{-1/(2\sqrt{2})}B^{\sqrt{2}+\ve}+B^\ve)$ curves of degree $O_\ve(1)$.
	This was proved before Salberger's work \cite{salb-cubic, salberger} emerged and  leads to a strictly weaker bound for $N_1(B)$.
	
	\begin{remark}
		It is interesting to think about possible lower bounds for $N_1(B)$. 
		Suppose that $a_1=a_2=1$ and $a_3=-c$, for some $c\in \NN$. On assuming that 
		$1\leq n\leq B^2/2$, we can focusing attention on solutions with $1\leq x_3\leq B /\sqrt{2c}$ in order to obtain the lower bound
		$$
		N_1(B)\geq \sum_{x\leq X} r(n+cx^2),
		$$
		where 
		$X=B /\sqrt{2c}$ and 
		$r(m)$ is the function that counts the number of representations of $m\in \NN$ as a sum of two squares.  By adapting work of Hooley \cite{hooley63} to the $r$-function, it should be possible to prove an asymptotic formula for the right hand side, as $X\to \infty$, although  one would need to 
		take care that everything is made sufficiently uniform in $n$ and $c$. In this way we expect that it is possible to prove a lower bound of the shape $N_1(B)\gg_\ve B/ c^{1/2+\ve}$, for any $\ve>0$, at
		least if $c$ is not too large in terms of $B$. 		The authors are very grateful to the referee for this remark. 
	\end{remark}
	
	\medskip
	
	We would also like to demonstrate that our work can be used to tackle
	higher-dimensional counting problems. In Theorem 
	\ref{thm:bound} we present a general upper bound for a class of threefolds
	in		$\mathbb{A}^4$, which we then apply to the affine equation
	$$
	x_1^k+x_2^{l}+x_3^{m}+x_4^k=N,
	$$
	for  non-zero $N\in \ZZ$ and suitable integer exponents $k,l,m\geq 2$. 
	Let 	
	$
	N_2(B)$ denote the number of 
	solutions 
	$\uu{x}\in \ZZ^4$ to this equation with  $|\uu{x}|\leq B$.
	Combining divisor bounds with uniform bounds for Thue equations, one can easily show 
	that 
	$$N_2(B)=O_\ve(B^{2+\ve}|N|^\ve).
	$$ The problem of improving this ``trivial bound''
	has a rich history, but most work has  focused on counting  
	representations of a large positive integer $N$  by  $\uu{x}\in \ZZ_{\geq 0}^4$, in which case it is natural to take $B=N^{1/k}$, when $k=l=m$.
	Let $N_{2}^+(B)$ denote this counting function. 
	Then an old (and still unbeaten) result of 
	Hooley \cite{hooley} provides the upper bound 
	$N_2^+(B)=O_\ve(B^{2-1/6+\ve})$
	in the case $k=l=m=3$. Hooley used sieve methods, an approach that Wisdom 
	extended to handle  
	the case $k=l=m=5$ in
	\cite{wisdom'}, and the cases 
	$k=l=3$ and $m\in \{4,5\}$  in \cite{wisdom}.
	More recently, Salberger's version of the determinant method has  also been applied to this problem by
	Marmon \cite{marmon}, with the outcome 
	that 
	$
	N_2^+(B)\ll_\ve B^{1+2/\sqrt{k}+\ve},
	$
	when $k=l=m$. 
	(Marmon gets better bounds when the implied constant is allowed to depend additionally on $N$, but our focus in this paper is on uniform bounds.)
	We shall deduce the following result from Theorem \ref{thm:side'} in Section \ref{s:unlike}.

	\begin{corollary}\label{thm:kl}
		Let $N\in \ZZ_{\neq 0}$ and let $k,l,m$ be integers such that $k\geq 13$ is odd and $k>l> m\geq 2$.
		Then, 
		for all $\ve>0$, we have 
		$$N_2(B)\ll_\ve B^{\frac{4}{3}+\frac{1}{\sqrt{k-1}}(1-\frac{1}{2l})+\ve}+B^{1+\frac{2}{\sqrt{k-1}}(1-\frac{1}{2l})+\ve}.
		$$
	\end{corollary}

	We shall see that the restriction $k\geq 13$ 
	allows us  to control the contribution from lines and conics in the hypersurface through work of  Newman and Slater \cite{waring}. The assumption that $k$ is odd allows us to factorise 
	$x_1^k+x_4^k$ over $\QQ$, which is a crucial first step in the proof.

	The determinant method approach of Marmon \cite{marmon} could also be used to give an upper bound for $N_2(B)$.   His approach is based on fixing a variable $x_3$, say, and then applying the 
	argument behind \eqref{eq:s1} to estimate the number of points that contribute to $N_2(B)$ on the resulting affine surface. 
	When carried through this  would lead to an alternative bound
	$$N_2(B)\ll_\ve B^{\frac 43+\frac{1}{\sqrt{k}}+\ve}+B^{1+\frac{2}{\sqrt{k}}+\ve},
	$$
	for $k\geq l\geq m\geq 2$, with $k\geq 13$. 
	While this is more efficient for  $k= l$, 
	Corollary \ref{thm:kl}  is  stronger for  $k>l$, since then  	
	$
	\frac{1}{\sqrt{k-1}}\left(1-\frac{1}{2l}\right)< \frac{1}{\sqrt{k}}.
	$

	\section{Salberger's global determinant method with a side condition}\label{s:main}

	We begin  by recalling the notation that was introduced in the  introduction. 
	Let $f\in \ZZ[x_1,x_2,x_3]$ be absolutely irreducible, let  $g\in \ZZ[x_2,x_3]$ and let $B_1,B_2,B_3\geq 2$.
	Let $\uu{m}\in \ZZ_{\geq 0}^3\setminus \{\uu{0}\}$ be the maximum among exponents appearing in $f$ for a fixed linear total order.
	Then we put $T_{\uu{m}}=\uu{B}^{\uu{m}}$ and we define $S=\max_{\uu{e}}\{\log \uu{B}^{\uu{e}}\}$, where the maximum runs over exponents $\uu{e}\in \ZZ_{\geq 0}^3$ such that $\uu{x}^{\uu{e}}$ appears in $g$ with a non-zero coefficient.  Let $B$ and $B'$ be given  by \eqref{eq:min-max}. 
	Let $\ve>0$.  
	It will be convenient to define 
	\begin{equation}\label{eq:defK}
		K_\ve=\exp\left(
		\sqrt{\frac{\prod_i \log B_i}{\log T_{\uu{m}}}}\left(1-\frac{m_1\log B_1}{2S\log T_{\uu{m}}}\log q\right)+\ve\log B\right),
	\end{equation} 
	so that $K_\ve=KB^\ve$ in the notation of  \eqref{eq:defK'}.

	Our proof of Theorems \ref{thm:side} and \ref{thm:side'} will be achieved by blending arguments of Heath-Brown \cite{cime} with the relevant arguments  of Salberger \cite{salberger}. 
	A singular point of $f (\uu{x}) = 0$ satisfies $\frac{\partial f}{\partial x_i}(\uu{x})=0$ for $1\leq i\leq 3$.
	Not all of these derivatives can vanish  identically on the zero locus of $f$, since $f$ is 
	absolutely irreducible. 
	Thus
	we can account 
	for all singular points if we include a non-zero partial derivative among the auxiliary polynomials 
	$f_1,\dots,f_J$.

	We proceed
	by introducing the exponent set that will play a  critical role in our argument. For $Y\geq 2$, let 
	$$
	\cE(Y)=\{\uu{e}\in \ZZ_{\geq 0}^3: \log \uu{B}^{\uu{e}}\leq Y, e_i<m_i \text{ for at least one } i\}
	$$
	and put	
	\begin{equation}\label{eq:pen}
		\cE^1(Y)=\{\uu{e}\in \cE(Y): e_1<m_1\}.
	\end{equation}
	Let 
	$h(\uu{x})=\sum_{\uu{e}\in\cE(Y)}a_{\uu{e}}\uu{x}^{\uu{e}}$,  for any coefficients $a_{\uu{e}}\in \ZZ$. We  proceed to show that  $h$ must be  
	coprime to $f$. Indeed, suppose for a contradiction that $h=fk$, for some further polynomial
	$k\in \bar\QQ[\uu{x}]$. 
	Let $\uu{k}$ be the maximum exponent (for the fixed linear total order) that appears in $k$ with non-zero coefficient. 
	Note that if $\uu{m'}\prec \uu{m}$ and $\uu{k'}\prec \uu{k}$ then $\uu{m'}+\uu{k'}\prec \uu{m}+\uu{k}$.
	Hence it follows that the exponent $\uu{m}+\uu{k}$ can be written in a unique way as the sum of an exponent in $f$ and an exponent in $k$,  since $\uu{m}$ and $\uu{k}$ are the maxima for the linear total order. Thus 
	$\uu{m}+\uu{k}$ must appear as an exponent in $h$ with  a non-zero coefficient, whence 
	$\uu{m}+\uu{k}\in \cE(Y)$. But $m_i+k_i\geq m_i$ for each $1\leq i\leq 3$, which  contradicts the 
	definition of $\cE(Y)$. Thus we cannot have that $h$ is a multiple of $f$ over $\bar \QQ$.

	We let $E=\#\cE(Y)$.
	For each $Z\geq \log T_{\uu{m}}$, 
	it follows from \cite[Eqs.~(5.9) and (5.10)]{cime} that 
	there exists $Y\in [Z,2Z]$ such that
	\begin{equation}\label{eq:defE}
		E=\frac{\log T_{\uu{m}}}{\prod_i \log B_i}\frac{Y^2}{2}\left(1+O\left(\sqrt{\frac{\log T_{\uu{m}}}{Y}}\right)\right)
	\end{equation}
	and 
	\begin{equation}\label{eq:BE}
		\sum_{\uu{e}\in \cE(Y)}\log \uu{B}^{\uu{e}}
		=\frac{\log T_{\uu{m}}}{\prod_i \log B_i}\frac{Y^3}{3}\left(1+O\left(\sqrt{\frac{\log T_{\uu{m}}}{Y}}\right)\right).
	\end{equation}
	We fix  such a $Y$ with $Y\gg \log B$, for a sufficiently large  implied constant that depends only on the degrees of $f$ and $g$. In particular,  we must then have $Y\gg \log T_{\uu{m}}$, since 
	$T_{\uu{m}}\leq B^{\deg f}$.
	
	Let $q\in \NN$.
	We henceforth write $X\subset \AAA^3$ for the geometrically irreducible surface defined by $f=0$. 
	We redefine $S_q(f,g;\uu{B})$ to focus on non-singular points of $X$, so that  $\nabla f(\uu{x})\neq \uu{0}$ for any $\uu{x}\in S_q(f,g;\uu{B})$.
	Let $r_1<\dots< r_u$ be a sequence of primes that are coprime with $q$ and such that the reduction modulo $r_i$ of $X$ is geometrically irreducible. 
	Let 
	\begin{equation}\label{eq:def-r}
		r=
		\begin{cases} 
			r_1\cdots r_u & \text{if $u\geq 1$,}\\
			1 & \text{if $u=0$.}
		\end{cases}
	\end{equation}
	For each $i\leq u$, let $\uu{\eta}_i\in X(\F_{r_i})$ be a non-singular point of the reduction of $X$ modulo $r_i$. Let $\uu{\eta}=(\uu{\eta}_1,\dots, \uu{\eta}_u)$ and put
	\begin{equation}\label{eq:srt}
		S_{r,\uu{\eta}}=\left\{
		\uu{x}\in X(\ZZ): 
		\begin{array}{l}
			\uu{x}\equiv \uu{\eta}_i\bmod{r_i} ~\forall i\leq u\\
			|x_1|\leq B_1, |x_2|\leq B_2, |x_3|\leq B_3\\
			g(\uu{x})\equiv 0\bmod q
		\end{array}
		\right\}.
	\end{equation}
	Let $J=\#S_{r,\uu{\eta}}$ and let 
	$\uu{x}_1,\dots,\uu{x}_J$ be the elements of $S_{r,\uu{\eta}}$.  
	We may then  form the $J\times E$ matrix 
	\begin{equation}\label{eq:MM}
		\mathbf{M}=\left(\uu{x}_j^{\uu{e}}\right)_{
			\substack{\uu{e}\in \mathcal{E}(Y)\\
				1\leq j\leq J}},
	\end{equation}
	with the goal being to prove that 
	$
	\rank \mathbf{M}<E
	$
	under suitable choices for $r$ and $Y$.
	To achieve this, we may clearly  proceed under the assumption that $J\geq E$ and, on 
	relabelling, we may work with 
	the principal  
	$E\times E$ 
	minor
	\begin{equation}\label{eq:defDelta}
		\Delta=\abs{\det 
			\left(\uu{x}_j^{\uu{e}}\right)_{
				\substack{\uu{e}\in \mathcal{E}(Y)\\
					1\leq j\leq E}}}.
	\end{equation}
	Our goal is now to show that $\Delta=0$, which we shall achieve by comparing upper bounds for $\Delta$ with information about the $p$-adic valuation of $\Delta$
	for many primes $p$.
	
	Note that 
	$|\uu{x}_j^{\uu{e}}|\leq B_1^{e_1}B_2^{e_2}B_3^{e_3}$,  for each $1\leq j\leq J$. Hence it follows from 
	\eqref{eq:defE} and 
	\eqref{eq:BE} that 
	\begin{equation}\label{eq:delta}
		\log \Delta\leq \log E^E+
		\sum_{\uu{e}\in \cE(Y)}\log \uu{B}^{\uu{e}}
		=\frac{\log T_{\uu{m}}}{\prod_i \log B_i}\frac{Y^3}{3}\left(1+O\left(\sqrt{\frac{\log T_{\uu{m}}}{Y}}\right)\right).
	\end{equation}
	Before turning to the $p$-adic valuation of $\Delta$, it will be convenient to record the estimate
	\begin{equation}\label{eq:E32}
		Y^3=2\sqrt{2}E^{3/2}\left(\frac{\prod_i\log B_i}{\log T_{\uu{m}}}\right)^{3/2}\left(1+O\left(\sqrt{\frac{\log T_{\uu{m}}}{Y}}\right)\right),
	\end{equation}
	which easily follows from \eqref{eq:defE} and the assumption that $Y\gg \log T_{\uu{m}}$, for a sufficiently large implied constant.

	We are now ready to collect together estimates for the divisibility of $\Delta$ by various prime powers. 
	Let $r_1<\cdots<r_u$ be the primes that were chosen above and recall the definition 
	\eqref{eq:def-r} of $r$. 
	The following result is due to Heath-Brown  \cite[Eq.~(5.11)]{cime}, combined with 
	the estimate in \eqref{eq:E32} for $Y^3$.

	\begin{lemma}\label{lem:p}
		There exists $\nu\in \NN$ such that $r^\nu\mid \Delta$, with 
		$$
		\nu=\frac{2\sqrt{2}E^{\frac 32}}{3}\left(1+O\left(\sqrt{\frac{\log T_{\uu{m}}}{Y}}\right)\right).
		$$
	\end{lemma}
	
	The   following result is due to Salberger
	\cite[Lemmas~1.4 and  1.5]{salberger}.
	
	\begin{lemma}\label{lem:otherprimes}
		Let $p$ be a prime  and assume that the reduction modulo $p$ of $X$ is geometrically integral. There exists $N_p\in \NN$ such that $p^{N_p}\mid \Delta$, with 
		$$N_p\geq \frac{2\sqrt{2}}{3}\frac{E^{3/2}}{p}+O\left(E+E^{3/2}p^{-3/2}\right).
		$$
	\end{lemma}

	We shall need a good lower bound for the  average of  $N_p$, weighted by $\log p$.
	
	\begin{lemma}\label{lem:sumother}
		Let $\pi_X$ be the product of primes $p$ such that the reduction modulo $p$ of $X$ is not geometrically integral. 
		Let $N_p$ be as in Lemma \ref{lem:otherprimes}, for any prime $p$. Then 
		$$
		\sum_{\substack{p\leq x\\p\nmid qr\pi_X}}N_p\log p\geq \frac{2\sqrt{2}}{3}E^{3/2}\left(\log x+O\left(\frac{x}{\sqrt{E}}+\log\log 3qr\pi_X\right)\right).
		$$
	\end{lemma}
	\begin{proof}
		It follows from Lemma \ref{lem:otherprimes} that 
		\begin{align*}
			\sum_{\substack{p\leq x\\p\nmid qr\pi_X}}N_p\log p&\geq 
			\sum_{\substack{p\leq x\\p\nmid qr\pi_X}}
			\left(
			\frac{2\sqrt{2}}{3}\frac{E^{3/2}\log p}{p}+O\left(E\log p+\frac{E^{3/2}\log p}{p^{3/2}}\right)\right)\\
			&= \frac{2\sqrt{2}}{3}E^{3/2}\sum_{\substack{p\leq x\\p\nmid qr\pi_X}}\frac{\log p}{p}+O\left(Ex+E^{3/2}\right),
		\end{align*}
		by the prime number theorem.
		According to  \cite[Lemma~1.10]{salberger}, we have 
		$$
		\sum_{\substack{p\leq x\\p\nmid qr\pi_X}}\frac{\log p}{p}=\sum_{\substack{p\leq x}}\frac{\log p}{p}+O(\log\log 3qr\pi_X)=\log x+O(\log\log 3qr\pi_X).
		$$
		Hence the statement of the lemma follows.
	\end{proof}
	
	Our remaining task is to  show that $q^\lambda$ divides $\Delta$ for suitable $\lambda\in \NN$, using the congruence condition  in 
	$S_q(f,g;\uu{B})$.
	This is the part of our paper that  relies crucially on the fact that  $g$ does not depend on $x_1$.
	Indeed, our approach will use  the fact that all the monomials appearing in $g(x_2,x_3)^{m}$ belong to $\cE(Y)$, provided that  $m$ is sufficiently small, a property that 
	could fail if $g$ were to depend on $x_1$,  since it could obstruct the constraint that  $e_i<m_i$ for at least one index $i\in \{1,2,3\}$.
	
	Recall the definition \eqref{eq:pen} of the set 
	$ \cE^1(Y)$. For any $\uu{e}\in  \cE^1(Y)$ and $\uu{t}\in \ZZ_{\geq 0}^3$, let 
	$\lambda_{\uu{e},\uu{t}}$ be the largest integer $\lambda\in \ZZ_{\geq 0}$ 
	such that $\uu{e}$ can be written as $\uu{e}=\uu{e}'+\lambda \uu{t}$,  for some $\uu{e}'\in \cE^1(Y)$. 
	Note that 
	$0\leq \lambda_{\uu{e},\uu{t}}\leq \max\{e_1,e_2,e_3\}$ if 
	$\uu{t}\neq \uu{0}$. If 
	$\uu{t}= \uu{0}$, then we put $\lambda_{\uu{e},\uu{t}}=\infty$.
	
	In the proof of the next result, we shall  perform elementary column operations to the columns of a matrix, and we proceed to  explain the procedure briefly. Let $\M$ be an $E\times E$ matrix with columns  $\{v_i\}_{1\leq i\leq E}$. At the first step, we substitute $v_1$ by a vector of the form $\sum_{1\leq i\leq E} a_i^{(1)}v_i$, for suitable integers $a_1^{(1)},\dots, a_E^{(1)}$, with $a_1^{(1)}=1$. In the same way, at the $k$-th step, we substitute $v_k$ by a vector $\sum_{1\leq i\leq E} a_i^{(k)}v_i$, with $a_k^{(k)}=1$ and $a_i^{(k)}=0$ for $i<k$. After $E$ steps, we will obtain a matrix 
	$\M^{(E)}$, say. 
	For $1\leq k\leq E$, 	the  $k$-th column of $\M^{(E)}$  is equal to $\sum_{1\leq i\leq E} a_i^{(k)}v_i$, 
	with $a_k^{(k)}=1$ and $a_i^{(k)}=0$ for $1\leq i<k$. 
	Clearly  $\M^{(E)}=\mathbf{A}\M$, where $\{A_{k,i}\}_{1\leq i,k\leq E}=a_i^{(k)}$. In particular,  we have $\det(\M^{(E)})=\det(\M)$, since  $\mathbf A$ is lower triangular and the entries of the diagonal are all equal to $1$.
	
	\begin{lemma}\label{lem:mue}
		Let $\uu{t}=(0,t_2,t_3)$ be such that the coefficient of $\uu{x}^{\uu{t}}$ in $g$ is coprime to $q$.  Assume that $\uu{t}$ is the maximum or the minimum, among the monomials that appear in $g$ with a non-zero coefficient, for a strict total order $\prec'$ on $\ZZ_{\geq 0}^3$. Then $q^\lambda\mid \Delta$, where
		\begin{equation}\label{eq:deflambda}
			\lambda=
			\sum_{\uu{e}\in \cE^1(Y)}	
			\min\left\{\lambda_{\uu{e},\uu{t}},\left\lfloor \frac{Y-\log \uu{B}^{\uu{e}}}{S-\log \uu{B}^{\uu{t}}}\right\rfloor\right\}.
		\end{equation}
	\end{lemma}
	\begin{proof}
		Let $c_{\uu{t}}$ denote the coefficient of  $\uu{x}^{\uu{t}}$ in $g$, which is assumed to be coprime to $q$. Let 
		$z\in \ZZ$ such that $zc_{\uu{t}}\equiv 1\bmod q$ and write $s=(zc_{\uu{t}}-1)/q\in \ZZ$. 
		Further, let $g_{q}(\uu{x})=zg(\uu{x})-qs\uu{x}^{\uu{t}}$. Then $g_{q}\in \ZZ[x_2,x_3]$ and 
		the coefficient of $\uu{x}^{\uu{t}}$ is $1$. It is clear that 
		\begin{equation}\label{eq:side}
			g_{q}(\uu{x})\equiv 0\bmod q
		\end{equation} 
		for any $\uu{x}\in S_{r,\uu{\eta}}$.

		Fix a choice of $\uu{e}\in \cE^1(Y)$ and set
		\begin{equation}\label{eq:mu-def}
			\mu_{\uu{e},\uu{t}}=
			\min\left\{\lambda_{\uu{e},\uu{t}},\left\lfloor \frac{Y-\log \uu{B}^{\uu{e}}}{S-\log \uu{B}^{\uu{t}}}\right\rfloor\right\}.
		\end{equation}
		Let $\uu{e}'=\uu{e}-\lambda_{\uu{e},\uu{t}}\uu{t}$, where
		$\lambda_{\uu{e},\uu{t}}\in \ZZ_{\geq 0}$ is defined before the statement of the lemma. Let $$g_q^{(e)}(\uu{x})=\uu{x}^{\uu{e'}+(\lambda_{\uu{e},\uu{t}}-\mu_{\uu{e},\uu{t}})\uu{t}}g_q(\uu{x})^{\mu_{\uu{e},\uu{t}}}.$$
		Notice that $\lambda_{\uu{e},\uu{t}}-\mu_{\uu{e},\uu{t}}\geq 0$ and that 
		any  monomial in $g_q^{(e)}(\uu{x})$ must have exponent vector belonging to 
		$\cE^1(Y)$, since the exponent of $x_1$ is $e_1'=e_1<m_1$ and 
		$$
		\log \uu{B}^{\uu{e'}+(\lambda_{\uu{e},\uu{t}}-\mu_{\uu{e},\uu{t}})\uu{t}}+\mu_{\uu{e},\uu{t}}S=\log \uu{B}^{\uu{e}}+\mu_{\uu{e},\uu{t}}(S-\log \uu{B}^{\uu{t}})\leq Y.
		$$
		Moreover, we have 
		$g_q^{(e)}(\uu{x})\equiv 0\bmod{q^{\mu_{\uu{e},\uu{t}}}}$ 
		for all $\uu{x}\in S_{r,\uu{\eta}}$,
		by \eqref{eq:side}.
		
		Assume that $\uu{t}$ is the minimum for the strict total order $\prec'$ on $\ZZ_{\geq 0}^3$. Notice that the coefficient of $\uu{x}^{\uu{e}}$ in $g_q^{(e)}(\uu{x})$ is $1$, since the coefficient of $\uu{x}^{\uu{t}}$ in $g_{q}(\uu{x})$ is $1$. We index the exponents in $\cE(Y)$ in a such a way that $\uu{e}^{(1)}\prec'\uu{e}^{(2)}\prec'\dots \prec'\uu{e}^{(E)}$. Letting $v_j^{(i)}=\uu{x}_j^{\uu{e}^{(i)}}$, it follows from  \eqref{eq:defDelta} that 
		$$
		\Delta=\abs{\det 
			\left(\uu{x}_j^{\uu{e}^{(i)}}\right)_{
				\substack{
					1\leq i,j\leq E}}}=\abs{\det 
			\left(v^{(i)}\right)_{
				\substack{
					1\leq i\leq E}}}.
		$$
		
		We proceed by
		performing elementary column operations to the column $v^{(k)}$, starting from $k=1$ to $k=E$, for all $k$ such that $\uu{e}^{(k)}\in\cE^1(Y)$. Notice that the coefficient of $\uu{x}^{\uu{e}^{(k)}}$ 
		in $g_{q}^{(\uu{e}^{(k)})}$ is $1$ and it is the smallest monomial (for the order $\prec'$) appearing in the polynomial.
		On adding every other monomial of
		$g_{q}^{(\uu{e}^{(k)})}(\uu{x})$, 
		we may assume that  the column takes the form
		$g_{q}^{(\uu{e}^{(k)})}(\uu{x})$. Notice that we have substituted $v^{(k)}$ with $\sum_{1\leq k\leq E} a_i^{(k)}v^{(i)}$, with $a_k^{(k)}=1$ and $a_i^{(k)}=0$ for $i<k$. Indeed, every monomial in $g_{q}^{(\uu{e}^{(k)})}(\uu{x})$ is larger for the order than $\uu{x}^{\uu{e}^{(k)}}$ and so it corresponds to a column with larger index. As we explained before the statement of the lemma, this process does not change the value of $\Delta$.
		Each term of the $k$-th column is divisible by $q^{\mu_{\uu{e}^{(k)},\uu{t}}}$. The statement of the lemma easily follows.
		
		If we assume that $\uu{t}$ is the maximum for the strict total order $\prec'$ on $\ZZ_{\geq 0}^3$, we index the exponents in $\cE(Y)$ in a such a way that $\uu{e}^{(E)}\prec'\uu{e}^{(E-1)}\prec'\dots \prec'\uu{e}^{(1)}$ and the lemma follows in the exact same way.
	\end{proof}
	
	Our next task is to get a useful bound 
	for 
	$\lambda$ in some  particular cases. In what follows, it will be convenient to set 
	\begin{equation}\label{def:R}
		R=\frac{m_1(\log B_1)^{\frac 32}(\log B_2)^{\frac 12}(\log B_3)^{\frac 12}}{2S(\log T_{\uu{m}})^{\frac32}},
	\end{equation}
	where $T_{\uu{m}}$ and $S$ were defined in  \eqref{def:T} and \eqref{def:S}, respectively.
	In the light of \eqref{eq:defK}, we clearly have 
	\begin{equation}\label{eq:lid}
		\log K_\ve=\sqrt{\frac{\prod_i \log B_i}{\log T_{\uu{m}}}}-R\log q+\ve\log B.
	\end{equation}
	
	We are going to apply Lemma \ref{lem:mue} in the case $\uu{t}=(0,0,0)$,  or the case $\uu{t}\in \{(0,t_2,0), (0,0,t_3)\}$ with the additional assumption that $S=\log\uu{B}^{\uu{t}}$. In the first case, $\uu{t}$ is the minimum, among the monomials that appear in $g$ with a non-zero coefficient, for the lexicographic order. If $\uu{t}=(0,t_2,0)$, then $\uu{t}$ is the maximum for the order $(e_1,e_2,e_3)\prec'(e_1',e_2',e_3')$ if $e_2<e_2'$, or $e_2=e_2'$ and $e_3<e_3'$, or $e_2=e_2'$ and $e_3=e_3'$ and $e_1<e_1'$, since we are assuming $S=\log\uu{B}^{\uu{t}}$. In the same way, if $\uu{t}=(0,0,t_3)$, then $\uu{t}$ is the maximum for a strict total order. Thus, in all these cases, we can apply Lemma \ref{lem:mue}.

	\begin{lemma}\label{lem:ineqlambda2}
		Let $\uu{t}\in \{(0,0,0),(0,t_2,0), (0,0,t_3)\}$ and let $\lambda$ be given by \eqref{eq:deflambda}. If $\uu{t}\neq (0,0,0)$, furthermore,  assume that $S=\log\uu{B}^{\uu{t}}$.
		Then
		$$
		\lambda= \frac{2\sqrt{2}RE^{3/2}}{3}\left(1+O\left(\sqrt{\frac{\log B}{Y}}\right)\right),
		$$
		where $R$ is given by \eqref{def:R}.
	\end{lemma}
	\begin{proof}
		We begin by assuming that  $\uu{t}=(0,0,0)$, so that 
		$\lambda_{\uu{e},\uu{t}}=\infty$ and 
		$\log \uu{B}^{\uu{t}}=0$.
		We then have
		\begin{align*}
			\lambda=\sum_{\uu{e}\in\cE^1(Y)}\left\lfloor \frac{Y-\log\uu{B}^{\uu{e}}}{S} \right\rfloor
			&=\sum_{\uu{e}\in\cE^1(Y)}\left(\frac{Y-\log\uu{B}^{\uu{e}}}{S}+O(1)\right)
			\\&=\frac{Y}{S}\#\cE^1(Y)-\frac{1}{S}\sum_{\uu{e}\in\cE^1(Y)}\log\uu{B}^{\uu{e}}+O\left(\#\cE^1(Y)\right).
		\end{align*}
		Moreover,
		$$		\sum_{e_2\log B_2+e_3\log B_3\leq Y_1}1=\frac{Y_1^2}{2\log B_2\log B_3}\left(1+O\left(\frac{\log B}{Y_1}\right)\right),
		$$
		by  \cite[Eq.~(5.6)]{cime}. Thus we obtain
		\begin{align*}
			\#\cE^1(Y)&=\sum_{e_1<m_1}\sum_{e_2\log B_2+e_3\log B_3\leq Y-e_1\log B_1}1\\
			&=\sum_{e_1<m_1}\frac{(Y-e_1\log B_1)^2}{2\log B_2\log B_3}\left(1+O\left(\frac{\log B}{(Y-e_1\log B_1)}\right)\right)\nonumber\\&=\frac{m_1Y^2}{2\log B_2\log B_3}\left(1+O\left(\frac{\log B}{Y}\right)\right),
		\end{align*}
		since 
		$e_1\log B_1< m_1\log B_1\ll\log B\ll Y$.
		
		Next, we note that  
		\begin{align*}
			\sum_{e_2\log B_2+e_3\log B_3\leq Y_1}e_2&=\sum_{e_2\leq \frac{Y_1}{\log B_2}}\sum_{e_3\leq \frac{Y_1-e_2\log B_2}{\log B_3}}e_2
			\\
			&=\sum_{e_2\leq \frac{Y_1}{\log B_2}}\left\lfloor\frac{Y_1-e_2\log B_2}{\log B_3}+1\right\rfloor e_2\\
			&=\sum_{e_2\leq \frac{Y_1}{\log B_2}}\left(\frac{Y_1-e_2\log B_2}{\log B_3}\right) e_2
			+O\left(\frac{Y_1^2}{(\log B_2)^2}\right)\\&
			=\frac{Y_1^3}{6(\log B_2)^2\log B_3}+O\left(\frac{Y_1^2}{\log B_2\log B_3}+\frac{Y_1^2}{(\log B_2)^2}\right),
		\end{align*}
		since 
		$$
		\sum_{n\leq X}(X-n)n=\sum_{n\leq X}Xn-n^2=\frac{X^3}{2}-\frac{X^3}{3}+O(X^2)=\frac{X^3}{6}+O(X^2),
		$$
		for any $X>0$.
		But then it follows that 
		\begin{align}\label{eq:wizz1}
			\sum_{\uu{e}\in\cE^1(Y)}e_2
			=~&	\sum_{e_1<m_1}
			\sum_{\substack{e_2\log B_2+e_3\log B_3\leq Y-e_1\log B_1}}e_2\nonumber\\
			=~&\sum_{e_1<m_1}\frac{(Y-e_1\log B_1)^3}{6(\log B_2)^2\log B_3}
			+O\left(\frac{Y_1^2}{\log B_2\log B_3}+\frac{Y_1^2}{(\log B_2)^2}\right)
			\nonumber\\
			=~&
			\frac{m_1Y^3}{6(\log B_2)^2\log B_3}\left(1+O\left(\frac{\log B}{Y}\right)\right).
		\end{align}
		Similarly, 
		$$
		\sum_{\uu{e}\in\cE^1(Y)}e_3
		=\frac{m_1Y^3}{6\log B_2(\log B_3)^2}\left(1+O\left(\frac{\log B}{Y}\right)\right).	
		$$
		Hence we deduce that 
		\begin{align*}
			\sum_{\uu{e}\in\cE^1(Y)}\log\uu{B}^{\uu{e}}=~&\sum_{\uu{e}\in\cE^1(Y)}\left(e_2\log B_2+e_3\log B_3\right)+O\left(\#\cE^1(Y)\log B_1
			\right)\\
			=~&
			\frac{m_1Y^3}{3\log B_2\log B_3}\left(1+O\left(\frac{\log B}{Y}\right)\right).	\end{align*}
		Returning to our initial expression for $\lambda$, and applying  \eqref{eq:E32},
		we may now insert these estimates to conclude that 
		\begin{align*}
			\lambda&
			=\left(\frac{m_1Y^3}{2S\log B_2\log B_3}-\frac{m_1Y^3}{3S\log B_2\log B_3}\right)\left(1+O\left(\frac{\log B}{Y}\right)\right)\\&=\frac{m_1Y^3}{6S\log B_2\log B_3}\left(1+O\left(\frac{\log B}{Y}\right)\right)\\
			&=\frac{2\sqrt{2}m_1E^{3/2}(\log B_1)^{3/2}(\log B_2)^{1/2}(\log B_3)^{1/2}}{6S(\log T_{\uu{m}})^{3/2}}
			\left(1+O\left(\sqrt{\frac{\log B}{Y}}\right)\right).
		\end{align*}
		This concludes the proof of the lemma in 
		the case $\uu{t}=(0,0,0)$.

		We proceed to discuss the case  $\uu{t}=(0,t_2,0)$, with $t_2\neq 0$, the case $\uu{t}=(0,0,t_3)$ being identical. In this case we have $S-\log \uu{B}^{\uu{t}}=0$
		in  \eqref{eq:deflambda} and  $\lambda_{\uu{e},\uu{t}}=\lfloor e_2/t_2\rfloor=e_2/t_2+O(1)$, in the notation before Lemma~\ref{lem:mue}. But then it follows from \eqref{eq:wizz1} 
		that
		\begin{align*}
			\lambda= \sum_{\uu{e}\in \cE^1(Y)} \lambda_{\uu{e},\uu{t}}
			&=\frac{1}{t_2}\sum_{\uu{e}\in \cE^1(Y)} e_2+O(\#\cE^1(Y))\\
			&=
			\frac{m_1Y^3}{6t_2(\log B_2)^2\log B_3}\left(1+O\left(\frac{\log B}{Y}\right)\right).
		\end{align*}
		Recalling that  $S=t_2\log B_2$, we may therefore conclude 
		the statement of the lemma  by  applying  \eqref{eq:E32}, as previously. 	\end{proof}

	We are now ready to bring together our various  results about the $p$-adic valuation of $\Delta$ in order to record a lower bound for it.
	
	\begin{lemma}\label{lem:inD}
		Let $\ve>0$  and assume that $K_\ve\geq 1$, in the notation of 
		\eqref{eq:defK}. 	
		Assume that 
		$\gcd(q,g(0,0))=1$
		and 
		$\log r\pi_X\ll \log B$, and furthermore, that 
		$B\gg_\ve 1$ and  $Y\gg_\ve \log B$ for sufficiently large implied constants that depend only on $\ve$.  Then either $\Delta=0$ or 
		$$
		\log \Delta\geq \frac{2\sqrt{2}E^{3/2}}{3}\left(R\log q+\log \sqrt{E}r-\frac\ve2\log B\right).
		$$
	\end{lemma}
	\begin{proof}
		Since $\gcd(q,g(0,0))=1$, we may apply Lemma
		\ref{lem:mue} with $\uu{t}=(0,0,0)$ and get $q^{\lambda}\mid \Delta$.
		Combining this with Lemmas \ref{lem:p} and  \ref{lem:otherprimes}, 
		we therefore deduce that either $\Delta=0$ or
		\begin{align*}
			\log \Delta&\geq \lambda \log q+\nu\log r+\sum_{\substack{p\ll E^{1/2}\\p\nmid qr\pi_X}}N_p\log p.
		\end{align*}
		By Lemma \ref{lem:ineqlambda2}, we deduce
		that 
		$$
		\lambda=\frac{2\sqrt{2}RE^{3/2}}{3}\left(1+O\left(\sqrt{\frac{\log B}{Y}}\right)\right).
		$$
		Applying Lemma \ref{lem:sumother}, we conclude that there exists $E_1\geq 0$ such that 
		$$			\log \Delta			
		\geq 			
		\frac{2\sqrt{2}E^{3/2}}{3}\left(R\log q+\log \sqrt{E}r\right) -E_1,
		$$
		with 
		$$
		E_1\ll E^{3/2}\left(\sqrt{\frac{\log T_{\uu{m}}}{Y}}\log r+\log\log 3qr\pi_X+\sqrt{\frac{\log B}{Y}}R\log q\right)
		.
		$$
		By  \eqref{eq:lid}, we have
		$$
		R\log q=\left(\frac{\prod_i \log B_i}{\log T_{\uu{m}}}\right)^{\frac 12}+\ve \log B-\log K_\ve=O_\ve(\log B),
		$$
		since we are assuming that $K_\ve\geq 1$. Moreover, by \eqref{def:R} and the definitions of $S$ and $T_{\uu{m}}$ in \eqref{def:S} and \eqref{def:T}, we have
		\begin{align*}
			\log \log 3q=O_\ve(\log \log B)+\log( R^{-1})&=O_\ve(\log \log B+\log (S\log T_{\uu{m}}))\\ &=O_\ve(\log \log B).
		\end{align*}
		Hence
		\begin{align*}
			E_1&\ll_\ve E^{3/2}\left(\sqrt{\frac{\log B}{Y}}\log Br+\log\log 3qr\pi_X\right)
			\\
			&\ll_\ve E^{3/2}\left(\log\log B+\log B \sqrt{\frac{\log B}{Y}}\right), 
		\end{align*}
		since we have assumed that $\log r\pi_X\ll\log B$.
		On  assuming that $B$ and $Y/\log B$ are bigger than a sufficiently large constant depending only on $\ve$, we readily conclude that $E_1<\frac{\ve}{2}E^{3/2}\log B$, 
		from which the statement of the lemma is clear.
	\end{proof}

	Lemma \ref{lem:inD} is a critical ingredient in Theorem \ref{thm:side}. For Theorem \ref{thm:side'} we will need  the following variant.

	\begin{lemma}\label{lem:inD'}
		Let $\ve>0$  and assume that $K_\ve\geq 1$, in the notation of 
		\eqref{eq:defK}. 	Assume that $B_1=B_2=B_3=B$ and that 
		$q$ satisfies the hypothesis \eqref{eq:gcdq}. 
		Assume that $\log r\pi_X\ll \log B$, and furthermore, that 
		$B\gg_\ve 1$ and  $Y\gg_\ve \log B$ for sufficiently large implied constants that depend only on $\ve$.  Then either $\Delta=0$ or 
		$$
		\log \Delta\geq \frac{2\sqrt{2}E^{3/2}}{3}\left(R\log q+\log \sqrt{E}r-\frac\ve2\log B\right).
		$$
	\end{lemma}
	\begin{proof}
		Let us suppose that $l=\deg g$ and let  $p^j\| q$. 
		If $p\nmid g(0,0)$ then we deduce 
		from  Lemmas
		\ref{lem:mue} (applied to $p^j$ and not to $q$)  and \ref{lem:ineqlambda2}
		that 
		$p^{j\lambda}\mid \Delta$, where 
		$$
		\lambda=\frac{2\sqrt{2}RE^{3/2}}{3}\left(1+O\left(\sqrt{\frac{\log B}{Y}}\right)\right).
		$$
		If $p\mid g(0,0)$, on the other hand, then 
		\eqref{eq:gcdq} implies that 
		$p\nmid g_0(1,0)$ or $p\nmid g_0(0,1)$ and so we can apply 
		Lemma
		\ref{lem:mue}  with  $\uu{t}\in \{(0,l,0),(0,0,l)\}$.
		But then Lemma 
		\ref{lem:ineqlambda2} leads to precisely the same conclusion,
		since 
		$
		S= l\log B$. Thus we may conclude that $q^\lambda\mid \Delta$ and the remainder of the argument  proceeds exactly as in the proof of Lemma \ref{lem:inD}.
	\end{proof}
	
	Building on the previous two results we can prove the following result, in
	which we recall the notation 
	\eqref{eq:srt} for 
	$S_{r,\uu{\eta}}$. As we already mentioned at the beginning of the section, our proof is a blend of the determinant method developed by Salberger \cite{salberger} with the extra saving given by the side condition. The next result  is an analogue of \cite[Thm.~2.2]{salberger}.
	
	\begin{lemma}\label{prop:pol}
		Let $\ve>0$ and assume that $K_\ve\geq 1$,
		in the notation of 
		\eqref{eq:defK}. 
		Assume that $\gcd(q,g(0,0))=1$,  or that $B_1=B_2=B_3=B$ and 
		$q$ satisfies  \eqref{eq:gcdq}.
		Assume that  $\log \pi_X\ll \log B$ and  $B\gg_\ve1$, for a sufficiently large implied constant depending only on $\ve$.
		Let $r_1<\dots< r_u$ be a sequence of primes that are coprime with $q$ and such that the reduction modulo $r_i$ of $X$ is geometrically irreducible. 
		Let $r$ be given by \eqref{eq:def-r} and  assume that $\log r\ll \log B$. 
		For each $1\leq i\leq u$, let $\uu{\eta}_i\in X(\F_{r_i})$ be a non-singular point of the reduction modulo $r_i$ of $X$.
		Then there exists a polynomial $f_{\uu{\eta}}\in \ZZ[\uu{x}]$, which is coprime with $f$,  such that the points in $S_{r,\uu{\eta}}$ are all zeros of $f_{\uu{\eta}}$. Moreover, the polynomial has degree $O_\ve(
		\max\{1,K_\ve/r\}\log B/\log B')$. 
	\end{lemma}
	\begin{proof}
		Take $Y$ to be the smallest positive number such that $Y/\log B$ is  
		bigger than a constant depending only on $\ve$,  as in Lemmas \ref{lem:inD} or \ref{lem:inD'},  and such that $ \sqrt{E}r> K_\ve$.
		Let $\Delta$ be the determinant of an $E\times E$ minor of the matrix $\M$ as defined in \eqref{eq:MM}. 
		If $\Delta=0$ for all $E\times E$ minors,
		then $\rank\M< E$ and there is a $\ZZ$-linear combination $f_{\uu{\eta}}(\uu{x})=\sum_{\uu{e}\in\cE(Y)}a_{\uu{e}}\uu{x}^{\uu{e}}$ of monomials in $\cE(Y)$, such that 
		the points in $S_{r,\uu{\eta}}$ all vanish on  $f_{\uu{\eta}}$. 
		Moreover, as we proved at the beginning of the section, $f_{\uu{\eta}}$ will be coprime to $f$. Our choice of $Y$ ensures that 
		$$
		1<\frac{ \sqrt{E}r}{K_\ve}=\sqrt{\frac{\log T_{\uu{m}}}{\prod_i\log  B_i}} \cdot \frac{Y}{\sqrt{2}}\left(1+O\left(\sqrt{\frac{\log T_{\uu{m}}}{Y}}\right)\right)\frac{r}{K_\ve},$$
		by \eqref{eq:defE}.
		Since $Y$ is defined to be the smallest positive integer satisfying the constraints imposed in its definition, so it follows that $Y=O_{\ve}(\max\{1,K_\ve/r\}\log B)$. Finally,  the degree of $f_{\uu{\eta}}$ is bounded by $\lfloor Y/\log B'\rfloor$, whence 
		$$
		\deg f_{\uu{\eta}}\ll Y/\log B' \ll_{\ve}\max\{1,K_\ve/r\}\log B/\log B',
		$$
		as claimed.
		
		It remains to consider the possibility that $\Delta\neq 0$ for an $E\times E$ minor.		
		But then it follows from combining  \eqref{eq:lid} with  Lemmas \ref{lem:inD} or \ref{lem:inD'} that 
		\begin{align*}
			\log \Delta&\geq \frac{2\sqrt{2}E^{\frac 32}}{3}\left(R\log q+\log \sqrt{E}r-\frac\ve2\log B\right)\\&=\frac{2\sqrt{2}E^{\frac 32}}{3}\left(-\log K_\ve+\log \sqrt{E}r+\sqrt{\frac{\prod_i \log B_i}{\log T_{\uu{m}}}}+\frac\ve2\log B\right).
		\end{align*}
		However,  \eqref{eq:delta} and \eqref{eq:E32} combine to yield
		\begin{equation}\label{eq:knee}
			\log \Delta\leq \frac{2\sqrt{2}E^{\frac 32}}{3}\sqrt{\frac{\prod_i \log B_i}{\log T_{\uu{m}}}}\left(1+O\left(\sqrt{\frac{\log B}{Y}}\right)\right).
		\end{equation}
		Hence
		$$				\log \Delta
		\leq \frac{2\sqrt{2}E^{\frac 32}}{3}\left(\sqrt{\frac{\prod_i \log B_i}{\log T_{\uu{m}}}}
		+\frac\ve2\log B\right),
		$$
		which  contradicts the assumption $\sqrt{E}r>  K_\ve$.
	\end{proof}

	We now have everything in place to complete the proof of Theorems \ref{thm:side}
	and \ref{thm:side'}. 
	Let $\ve>0$ and 
	recall the notation $B'$ and $B$ from \eqref{eq:min-max}.
	Clearly, we may work under the assumption that $B\gg_\ve1$, for a sufficiently large implied constant that depends on $\ve$.
	
	We begin by dealing with the  case
	$K_\ve\leq1$, in the notation of \eqref{eq:defK}. (In particular, it will then  follow that 
	$KB^\ve\leq 1$ in \eqref{eq:defK'} and \eqref{eq:B1B'}.)
	In this case we fix $r=1$ and we choose  $Y\gg_\ve\log B$ such that \eqref{eq:BE} holds. Since $K_\ve \leq 1$, we deduce  from \eqref{eq:lid} that 
	$$
	R\log q\geq \sqrt{\frac{\prod_i \log B_i}{\log T_{\uu{m}}}}+\ve\log B.
	$$
	Let $\Delta$ be the determinant of an $E\times E$ minor of the matrix $\M$ defined in \eqref{eq:MM}. Here $r=1$,   so that $\uu{\eta}$ is the empty vector and 
	$$
	S_{1,\uu{\eta}}=\left\{
	\uu{x}\in X(\ZZ):
	|x_i|\leq B_i, ~g(\uu{x})\equiv 0\bmod q\right\}.
	$$
	The first few lines of the proofs of  Lemmas \ref{lem:inD} and \ref{lem:inD'} ensure that $q^\lambda\mid \Delta$,
	under the assumptions of 
	Theorems~\ref{thm:side}
	and \ref{thm:side'}, with 
	$$				\lambda=\frac{2\sqrt{2}RE^{\frac 32}}{3}\left(1+O\left(\sqrt{\frac{\log B}{Y}}\right)\right).
	$$
	Hence
	\begin{align*}
		\log \Delta 
		&\geq 
		\frac{2\sqrt{2}RE^{\frac 32}}{3}\left(1+O\left(\sqrt{\frac{\log B}{Y}}\right)\right)\log q\\
		&\geq \frac{2\sqrt{2}E^{\frac 32}}{3}\left(\sqrt{\frac{\prod_i \log B_i}{\log T_{\uu{m}}}}+\ve\log B\right)\left(1+O\left(\sqrt{\frac{\log B}{Y}}\right)\right),
	\end{align*}
	if  $\Delta\neq0$. This inequality is incompatible with the competing upper bound \eqref{eq:knee},
	if  $Y\gg_\ve \log B$ 
	for a sufficiently large implied constant. Hence $\Delta=0$ and
	we can construct an auxiliary polynomial 
	$f_{1}(\uu{x})=\sum_{\uu{e}\in\cE(Y)}a_{\uu{e}}\uu{x}^{\uu{e}}$, as in the proof of 
	Lemma~\ref{prop:pol}, which has all the properties required of it
	in the case $K_\ve\leq 1$.

	We are now ready to  deal 
	with the case $K_\ve\geq 1$. 
	If $\log \pi_X\ll \log B$ fails to  hold, then the desired conclusion follows from 
	\cite[Lemma~1.9]{salberger}. 
	Thus we may proceed under the assumption that $\log \pi_X\ll \log B$. But then  every assumption of Lemma \ref{prop:pol} holds and the proof is now identical to the one provided by Salberger \cite[Lemma~3.2]{salberger}.
	This completes the statement of 
	Theorem \ref{thm:side}
	and \ref{thm:side'}.
	
	\begin{remark}
		It is possible to prove a version of Theorem \ref{thm:side} in which the  assumption $\gcd(q,g(0,0))=1$ is replaced with the  weaker assumption 
		that $q$ is coprime to the content of $g$. 
		However, this would lead to a more complicated  
		expression for $K$ in  \eqref{eq:defK} and 
		will not be carried out here.
	\end{remark}

	\section{Counting points on quadric surfaces}\label{s:quadric}

	In this section we prove 
	Corollary \ref{thm:squares}, which concerns an upper bound for the counting function $N_1(B)$ defined in \eqref{eq:N1}.
	Let $S\subset \AAA^3$ be the quadric surface $$
	a_1x_1^2+a_2x_2^2+a_3x_3^2=n$$ 
	for $a_1,a_2,a_3,n\in \ZZ$ such that $-a_1a_2a_3n\neq \square$ and $\gcd(a_1,a_2,a_3,n)=1$.
	The condition 
	$-a_1a_2a_3n\neq \square$ precisely ensures that 
	$S$ contains no lines defined over $\QQ$. Indeed, if $X\subset \PP^3$ is the compactification of $S$, then $\Pic(X)\cong \ZZ^2$ if and only if $X(\QQ)\neq \emptyset$ and the underlying quadratic form has square determinant. Thus any line contained in $S$ cannot be defined over $\QQ$ and so contains at most one rational point.
	
	When $\max\{|n|,|\uu{a}|\}\gg B^{20}$, for a sufficiently large implied  constant, it follows from \cite[Lemma~5]{pila'} (which is based on an idea of Heath-Brown \cite[Thm.~4]{annal}),
	that there exists an auxiliary quadratic polynomial $q\in \ZZ[x_1,x_2,x_3]$, which is coprime to the polynomial defining $S$, and 
	such that the points counted by 
	$N_1(B)$ are also zeros of $q$. The intersection of $S$ with the quadric $q=0$ 
	produces at most $4$ irreducible curves and  Bombieri--Pila \cite[Thm.~5]{BP} implies that each irreducible curve  makes a contribution $O_\ve(B^{1/2+\ve})$, since any lines have only $O(1)$ points. 
	In this way it follows that 
	$$
	N_1(B)\ll_\ve B^{\frac{1}{2}+\ve} 
	$$
	if $\max\{|n|,|\uu{a}|\}\gg B^{20}$, which is satisfactory. We may therefore proceed under the assumption that $\max\{|n|,|\uu{a}|\}\ll B^{20}$.

	Suppose without loss of generality that $|a_1|=|\uu{a}|$. We shall rework Section 
	\ref{s:main} with 
	$q=|a_1|$ and $f(\uu{x})=a_1x_1^2+a_2x_2^2+a_3x_3^2-n$  and 
	$g(\uu{x})=a_2x_2^2+a_3x_3^2-n$. 
	Notice that $g(\uu{x})\equiv 0\bmod{q}$ for any $\uu{x}\in S(\ZZ)$.
	Moreover, we have 
	$$
	\gcd(q,g(0,0),g_0(1,0),g_0(0,1))=\gcd(a_1,n,a_2,a_3)=1. 
	$$
	Let $\pi_S=\abs{2a_1a_2a_3n}$ and note that the reduction modulo $p$ of $S$ is non-singular and geometrically integral for any $p\nmid \pi_S$. Moreover, we have $\log \pi_S\ll \log B$. 
	We shall take $B_1=B_2=B_3=B$ and 
	$\uu{m}=(2,0,0)$,
	noting that   $(2,0,0)$ is the maximum among the exponents appearing in $f$ with non-zero coefficient for the lexicographic order.
	
	The following result is extracted from work of Salberger \cite{salb-cubic} and will act as a proxy for  Lemma \ref{lem:p} in the present setting. 
	
	\begin{lemma}\label{lem:pv2}
		There exists $\nu\in \NN$ such that $r^\nu\mid \Delta$, with 
		$$
		\nu=E^{\frac 32}\left(1+O\left(E^{-1/2}\right)\right).
		$$
	\end{lemma}
	
	\begin{proof}
		When $B_1=B_2=B_3=B$, it is natural to take  $Y=\ell \log B$ in  
		the definition of $\mathcal{E}(Y)$, where $\ell$ is the maximum allowable degree of 
		the  monomials indexed by the set.
		In particular, we have $E=\#\mathcal{E}(Y)=Q(\ell)$, where
		$$
		Q(\ell)=\binom{\ell+3}{3}-\binom{\ell+1}{3}=\ell^2+O(\ell).
		$$
		On taking $N=3$ in the  remark after 
		Lemma~16.11 in  \cite{salb-cubic}, we deduce that $r^\nu\mid \Delta$, with 
		$
		\nu=\ell^3 +O(\ell^{2})$.  
		The statement of the lemma easily follows. 
	\end{proof}
	
	The next result is a refinement of 
	Lemma \ref{lem:otherprimes}.

	\begin{lemma}\label{lem:otherprimesv2}
		Let $p\nmid \pi_S$. There exists $N_p\in \NN$ such that $p^{N_p}\mid \Delta$, with 
		$$N_p\geq \frac{E^{3/2}}{p}+O\left(E+E^{3/2}p^{-3/2}\right).
		$$
	\end{lemma}
	
	\begin{proof}
		The proof is identical to the one provided by \cite[Lemmas~1.4 and  1.5]{salberger}, 
		but 	modified to 
		account for the new lower bound for the exponent that follows from  the remark after 
		\cite[Lemma 16.11]{salb-cubic}, which is valid since the reduction of $S$ modulo $p$ is non-singular for any $p\nmid \pi_S$.
	\end{proof}

	In what follows we take 
	\begin{equation}\label{eq:rice}
		K'=q^{-\frac{1}{6}} B^{\frac{2}{3}},
	\end{equation}
	recalling that $q=|\uu{a}|$.
	Let  $K_\ve'=K'B^{\ve}$.
	We are now ready to  prove our analogue of Lemma \ref{prop:pol} using the new bounds.
	
	\begin{lemma}\label{prop:polv2}
		Let $\ve>0$ and assume that $K'_\ve\geq 1$.
		Assume that  $B\gg_\ve1$, for a sufficiently large implied constant depending only on $\ve$.
		Let $r_1<\dots< r_u$ be a sequence of primes that are coprime with $q$ and such that the reduction modulo $r_i$ of $S$ is geometrically irreducible. 
		Let $r$ be given by \eqref{eq:def-r} and  assume that $\log r\ll \log B$. 
		For each $1\leq i\leq u$, let $\uu{\eta}_i\in S(\F_{r_i})$ be a non-singular point of the reduction modulo $r_i$ of $S$.
		Then there exists a polynomial $f_{\uu{\eta}}\in \ZZ[\uu{x}]$, which is coprime with $f$,  such that the points in $S_{r,\uu{\eta}}$ are all zeros of $f_{\uu{\eta}}$. Moreover, the polynomial has degree $O_\ve(
		\max\{1,K'_\ve/r\})$. 
	\end{lemma}
	\begin{proof}
		The proof is similar to  Lemma \ref{prop:pol}, but using Lemmas \ref{lem:pv2} and \ref{lem:otherprimesv2} instead of Lemmas \ref{lem:p} and \ref{lem:otherprimes}. We present the argument in full for the sake of clarity. 
		Since $B_1=B_2=B_3=B$ we are working with $Y=n\log B$ for a large parameter $n$. 
		We assume that $n\gg_\ve 1$ is the  smallest integer such that 
		$ \sqrt{E}r> K_\ve'$, where $E=Q(n)=n^2+O(n)$, as in the proof of Lemma \ref{lem:pv2}.
		
		Let $\Delta$ be the determinant of an $E\times E$ minor of the matrix $\M$ as defined in \eqref{eq:MM}. 
		If $\Delta=0$ for all $E\times E$ minors,
		then $\rank\M< E$ and there is a $\ZZ$-linear combination $f_{\uu{\eta}}(\uu{x})=\sum_{\uu{e}\in\cE(Y)}a_{\uu{e}}\uu{x}^{\uu{e}}$ of monomials in $\cE(Y)$, such that 
		the points in $S_{r,\uu{\eta}}$ all vanish on  $f_{\uu{\eta}}$. 
		Moreover, $f_{\uu{\eta}}$ will be coprime to $f$ and our  choice of $Y$ ensures that 
		$$
		1<\frac{ \sqrt{E}r}{K_\ve'}=n\left(1+O\left(\frac{1}{n}\right)\right)\frac{r}{K_\ve'}.$$
		Since $n$ is defined to be the smallest positive integer satisfying the constraints imposed in its definition, so it follows that 
		$
		\deg f_{\uu{\eta}} \leq n\ll_{\ve}\max\{1,K_\ve'/r\},
		$
		as claimed.
		
		It remains to consider the possibility that $\Delta\neq 0$ for an $E\times E$ minor.		
		In this case we 
		wish to apply the variant of 
		Lemma  \ref{lem:inD'} tailored to the situation at  hand.  We claim that 
		\begin{align*}
			\log \Delta&\geq 
			E^{3/2}\left(\frac{\log q}{6}+\log \sqrt{E}r-\frac\ve2\log B\right).
		\end{align*}
		To see this, we first apply 
		Lemma
		\ref{lem:ineqlambda2} to deduce that 
		$q^{\lambda}\mid \Delta$, with 
		$$
		\lambda=\frac{E^{3/2}}{6}\left(1+O\left(E^{-1/4}\right)\right),
		$$
		since $R=1/(4\sqrt{2})$ in \eqref{def:R} and $n\gg \sqrt{E}$.
		Combining this with 
		Lemmas \ref{lem:pv2} and~\ref{lem:otherprimesv2}, and then with the modified version of
		Lemma \ref{lem:sumother}, 
		we therefore deduce that 
		\begin{align*}
			\log \Delta&\geq \lambda \log q+\nu\log r+\sum_{\substack{p\ll E^{1/2}\\p\nmid qr\pi_S}}N_p\log p\\
			&\geq 	E^{3/2}\left(\frac{\log q}{6}+\log \sqrt{E}r\right) -E_1,	
		\end{align*}
		for some $E_1$ satisfying
		$$
		0\leq E_1\ll_\ve  E^{3/2}\left(E^{-1/4}\log q+E^{-1/2}\log r+\log\log 3qr\pi_S
		\right)
		.
		$$
		Clearly $E_1< \frac{\ve}{2}E^{3/2}\log B$ if $n\gg_\ve 1$, since  $\log q\ll \log B$ and 	
		we have assumed that $\log r\pi_S\ll\log B$.
		This establishes the claim. However, it follows from \eqref{eq:knee} that 
		$$				\log \Delta 
		\leq \frac{2E^{\frac 32}}{3}\left(\log B
		+\frac\ve2\log B\right),
		$$
		which  contradicts the assumption $\sqrt{E}r>  K_\ve'$.
	\end{proof}

	We are now ready to  conclude the proof of Corollary \ref{thm:squares}. Recall the definition 
	\eqref{eq:rice} of $K'$, in which we recall that $q=|\uu{a}|$.
	Our aim is to establish the existence of polynomials
	$f_1,\ldots, f_J\in \ZZ[x_1,x_2,x_3]$,  
	and a finite collection of points $Z\subset S$, such that 
	\begin{enumerate}
		\item $J=O_{\ve}(\max\{1,K'B^\ve\})$;
		\item
		each $f_j$ is coprime to $a_1x_1^2+a_2x_2^2+a_3x_3^2-n$ and has degree 
		$O_{\ve}(1)$, for $j\leq J$; 
		\item
		$\#Z=O_{\ve}(\max\{1,K'B^\ve\}^{2})$; and 
		\item
		for each $\uu{x}$ counted by $N_1(B)$, with $\uu{x}\not\in  Z$, 
		there exists $j\leq J$ such that
		$f_j(x_1,x_2,x_3)=0.
		$
	\end{enumerate}
	If $K'B^\ve\leq 1$ then the proof is identical to the corresponding case treated in the proof of  Theorem \ref{thm:side'}. Alternatively, if $K'B^\ve\geq 1$ 
	then the same proof goes through, making sure to use Lemma 
	\ref{prop:polv2}
	instead of Lemma~\ref{prop:pol}. 
	
	
	It remains to  handle the points on the $J$ curves contained in $S$.
	By Bombieri--Pila \cite[Thm.~5]{BP}, each irreducible curve  makes a contribution $O_\ve(B^{1/2+\ve})$, since the condition 
	$-a_1a_2a_3n\neq \square$ ensures that any lines contain $O(1)$ points. 
	Thus the overall contribution from the curves is $O_\ve(\max\{1,K'\}B^{1/2+2\ve})$ and 
	it follows that 
	$$
	N_1(B)\ll_\ve
	\frac{B^{\frac{4}{3}+\ve}}{|\uu{a}|^{1/3}}+\frac{B^{\frac{7}{6}+\ve}}{|\uu{a}|^{1/6}}+B^{\frac{1}{2}+\ve}.
	$$
	on redefining $\ve$.
	We claim that the first term can be dropped from this expression, which thereby yields 
	the statement of 
	Corollary \ref{thm:squares}. If the first term is less than second term we are done. Hence we may proceed under the assumption that the first term is greater than or equal to the second, which is equivalent to demanding that $B\geq |\uu{a}|$. But in this case we merely apply the divisor function bound 
	\eqref{eq:trivial}, which is less than the second term when 	$B\geq |\uu{a}|$.

	\section{Sums of unlike powers}\label{s:unlike}
	
	In this section we  prove  Corollary \ref{thm:kl}. In fact we shall begin by proving a more general result.

	\subsection{Counting points on affine threefolds}
	Consider the threefold $X\subseteq \AAA^4$ defined by the polynomial
	\begin{equation}\label{eq:3fold}
		f(x_1,x_2,x_3,x_4)=(\alpha x_1+\beta x_4+\gamma)h(x_1,x_2,x_3,x_4)+g(x_2,x_3),
	\end{equation}
	where $\alpha,\beta,\gamma\in \ZZ$ and $g,h$ are polynomials with integer coefficients, 
	with   $\deg(g)=l$ and 
	$\deg(h)=\deg_{x_1}(h)=k$. We further assume that $\beta\neq 0$. 
	We let  $X^\circ$ be the set of  points in $X$ that do not belong to a rational line or a rational conic contained in  $X$. We will use Theorem \ref{thm:side'} to bound the quantity
	$
	N^\circ(B)=\#\{\uu{x}\in X^\circ(\ZZ): |\uu{x}|\leq B\},$
	as follows.
	
	\begin{theorem}\label{thm:bound}
		Assume that the surface $X\cap \{\alpha x_1+\beta x_4+\gamma=q\}$ is geometrically irreducible
		for any non-zero $q\in \ZZ$.
		Assume that  $\gcd(g(0,0),g_0(1,0), g_0(0,1))=1$, where $g_0$ is the top degree part of $g$.
		Then, for any $\ve>0$, we have 
		$$
		N^\circ(B)\ll_\ve \max\{|\alpha|,|\beta|,|\gamma|\}B^\ve\left( B^{\frac{4}{3}+\frac{1}{\sqrt{k}}(1-\frac{1}{2l})}+B^{1+\frac{2}{\sqrt{k}}(1-\frac{1}{2l})}\right).
		$$
	\end{theorem}
	
	Before proving the theorem, we need a result from elementary number theory. 
	\begin{lemma}\label{lemma:tec}
		Given $-1<\alpha<0$ and $\ve>0$, we have
		$$
		\sum_{u\leq X}(u/\gcd(u,n))^\alpha\ll_\ve X^{\alpha+1}n^\ve.
		$$
	\end{lemma}
	\begin{proof}We have
		\begin{align*}
			\sum_{u\leq X}(u/\gcd(u,n))^\alpha&=\sum_{d\mid n}\sum_{\substack{u\leq X\\\gcd(u,n)=d}}(u/d)^\alpha \\
			&\leq \sum_{d\mid n}\sum_{\substack{u\leq X/d}}u^\alpha \\
			&\ll \sum_{d\mid n} \left(\frac{X}{d}\right)^{\alpha+1}\\
			&\ll n^\ve X^{\alpha+1},
		\end{align*}
		by the divisor  bound.
	\end{proof}
	\begin{proof}[Proof of Theorem \ref{thm:bound}]
		Let $L=\max\{|\alpha|,|\beta|,|\gamma|\}$ and let  $u\in \ZZ$. We write 
		$$
		N_u^\circ(B)=\#\{\uu{x}\in X^\circ(\ZZ): |\uu{x}|\leq B, 
		\alpha x_1+\beta x_4+\gamma=u\}.
		$$
		We have  $N_u^\circ(B)=0$ unless  $|u|\leq  3LB$, whence
		\begin{equation}\label{eq:ink}
			N^\circ(B)\leq \sum_{\abs{u}\leq 3LB} N_u^\circ(B).
		\end{equation}
		The points in $X$ with  $u=0$ are the points in $\{\alpha x_1+\beta x_4+\gamma=0\}\cap \{g(x_2,x_3)=0\}$. 
		We can estimate the number of relevant $(x_2,x_3)$ using Bombieri--Pila \cite[Thm.~5]{BP}. There are $O_\ve(B^{1/3+\ve})$ points that contribute, as we are not interested in points that lie on rational lines or conics. Since there are $O(B)$ choices of $(x_1,x_4)$ such that 
		$\alpha x_1+\beta x_4+\gamma=0$, this therefore leads to the bound
		\begin{equation}\label{eq:ink'}
			N_0^\circ(B)\ll_\ve B^{4/3+\ve}.
		\end{equation}
		
		Next assume that  $u\neq 0$. It will be convenient to put $\deg_{x_4}(h)=k'$ and  to write
		$$
		h_u(x_1,x_2,x_3)=\beta^{k'}h(x_1,x_2,x_3,(u-\alpha x_1-\gamma)/\beta).
		$$ 
		Since  $\beta\neq 0$, we see that $h_u$ is a  polynomial with integer coefficients and 
		$$
		\beta^{k'}h(x_1,x_2,x_3,x_4)=h_u(x_1,x_2,x_3),
		$$ 
		for any $\uu{x}$ counted by $N_u^\circ(B)$.
		Let $X_u\subseteq \AAA^3$ be the surface defined by the equation
		$$
		uh_u(x_1,x_2,x_3)+\beta^{k'}g(x_2,x_3)=0.
		$$
		This is geometrically irreducible by   assumption.
		Moreover,
		$$
		N_u^\circ(B)\leq\#\{\uu{x}\in X_u^\circ(\ZZ): |\uu{x}|\leq B\},
		$$
		where  $X_u^\circ$ is  the set of  points in $X_u$ that do not belong to any  rational lines or conics in 
		the surface. 
		Put $q=|u|/\gcd(u,\beta^{k'})$ and note that 
		$g(x_2,x_3)\equiv 0\bmod u$ for any $(x_1,x_2,x_3)\in X_u(\ZZ)$. 
		Moreover, the hypothesis \eqref{eq:gcdq} follows from the assumptions  of the theorem. 
		Thus 
		$$
		N_u^\circ(B)\leq \#\{\uu{x}\in X_u^\circ(\ZZ): |\uu{x}|\leq B,~g(x_2,x_3)\equiv 0\bmod q\}
		$$
		and we can apply Theorem \ref{thm:side'} to estimate this quantity.

		We can take $\uu{m}=(k,0,0)$ since $\deg (h)=\deg_{x_1} (h)=k$. Indeed, $(k,0,0)$ is the  maximum among the exponents appearing in the polynomial defining $X_{u}$ for the lexicographic order.
		Thus 
		$
		K=B^{\frac{1}{\sqrt{k}}}q^{-\frac{1}{2\sqrt{k}l}}
		$
		in \eqref{eq:B1B'}.
		Then, for any $\ve>0$, 
		it follows that there is a set of 
		$J=O_\ve(\max\{1,KB^\ve\})$ geometrically 
		irreducible curves $C_1,\dots,C_J\subset X_u$, with  $\deg C_j= O_\ve(1)$, such that 
		there are $O_\ve(\max\{1,KB^\ve\}^2)$  points counted by 
		$N_u^\circ(B)$ that are not  contained in  any of these curves. 
		When $\deg C_j\in \{1,2\}$, furthermore, we may assume that the relevant curve contains 
		$O(1)$ integer points, since we are not interested in lines or conics  defined over $\QQ$.
		The overall  contribution from the curves is therefore $O(JB^{1/3+\ve})$, by 
		\cite[Thm.~5]{BP}, which thereby shows that 
		\begin{align*}
			N_u^\circ(B)&\ll_\ve 
			\max\{1,KB^\ve\}B^{\frac 13+\ve}+\max\{1,KB^\ve\}^2\\
			&\leq  B^{\frac 13+\frac{1}{\sqrt{k}}+2\ve}q^{-\frac{1}{2\sqrt{k}l}}+B^{\frac{2}{\sqrt{k}}+2\ve}q^{-\frac{1}{\sqrt{k}l}}+B^{\frac 13+\ve}.
		\end{align*}
		Combining this with \eqref{eq:ink'} in \eqref{eq:ink}, we conclude that 
		\begin{align*}
			N^\circ(B)
			&\ll_\ve LB^{\frac 43+\ve}+\sum_{\substack{0< u\leq3 LB
			}} \left(
			\frac{
				B^{\frac 13+\frac{1}{\sqrt{k}}+2\ve}}{(
				u/\gcd(u,\beta^{k'}))^{\frac{1}{2\sqrt{k}l}}}+\frac{B^{\frac{2}{\sqrt{k}}+2\ve}}{
				(
				u/\gcd(u,\beta^{k'}))^{\frac{1}{\sqrt{k}l}}}\right).
		\end{align*}
		The statement of the theorem now follows from   Lemma \ref{lemma:tec} and redefining the choice of $\ve$.
	\end{proof}
	
	\subsection{Proof of  Corollary \ref{thm:kl}}
	We are now ready to prove
	Corollary \ref{thm:kl}.  Let $N\in \ZZ$ be non-zero. 
	Since $k$ is assumed to be odd, we can write
	$$
	x_1^k+x_2^{l}+x_3^{m}+x_4^k-N=(x_1+x_4)h(x_1,x_4)+x_2^{l}+x_3^{m}-N,
	$$
	for a polynomial $h\in \ZZ[x_1,x_4]$ of degree  $k-1$. 
	Let $$
	f(\uu{x})=(x_1+x_4)h(x_1,x_4)+x_2^{l}+x_3^{m}-N.
	$$
	This is a polynomial of the form \eqref{eq:3fold}, with $g(x_2,x_3)=x_2^{l}+x_3^{m}-N$.
	In particular, $$\gcd(g(0,0),g_0(1,0),g_0(0,1))=1.$$
	We need to check that $X_u=X\cap\{x_1+x_4=u\}$ is geometrically irreducible for all non-zero $u\in \ZZ$. To do so, we pick $n\in \QQ$ such that $M=N-uh(n,u-n)$ is non-zero.
	The polynomial 
	$x_2^{l}+x_3^{m}-M$ is absolutely irreducible of degree $l>m\geq 2$, as one can see by applying the Eisenstein criterion over $\bar{\QQ}[x_3]$. Indeed, 
	the polynomial $x_3^{m}-M\in\bar{\QQ}[x_3]$ is square-free, since it has non-zero discriminant. 
	Having checked all the hypotheses of Theorem \ref{thm:bound}, it now follows that the cardinality of points that do not lie on any   lines or   conics contained in $X$  
	is 
	$$
	\ll_\ve \left( B^{\frac{4}{3}+\frac{1}{\sqrt{k-1}}(1-\frac{1}{2l})}+B^{1+\frac{2}{\sqrt{k-1}}(1-\frac{1}{2l})}\right)B^\ve,
	$$
	since $l> m$. 
	This is satisfactory for Corollary 
	\ref{thm:kl}.

	It remains to deal with the  points that lie on lines and conics contained in the hypersurface $f=0$. 
	For this we can  adapt an argument of Newman and Slater \cite{waring}.

	\begin{lemma}\label{lemma:wron}
		Let $b\in \bar\QQ^*$,
		let 
		$l_1,\dots,l_r\in \NN$ and let 
		$\gamma_1(t),\dots,\gamma_r(t)\in \overline{\QQ}[t]$ such that
		\begin{equation}\label{eq:gamman}
			\sum_{i=1}^r\gamma_i^{l_i}=b.
		\end{equation}
		Put   $d_i=\deg(\gamma_i)$ for $1\leq i\leq r$ and 
		assume that $\gamma_1^{l_1},\dots,\gamma_r^{l_r} $ are linearly independent over $\bar \QQ$. Then
		$$
		\max_{1\leq i\leq r} d_il_i
		\leq (r-1)(d_1+\dots+d_r)-\frac{r(r-1)}{2}.
		$$
	\end{lemma}
	\begin{proof}
		Let $e_i=d_il_i$ for $1\leq i\leq r$. Let $s_i=\max\{l_i-r+1,0\}$.
		Let $W$ be the Wronskian 
		$W(\gamma_1^{l_1},\gamma_2^{l_2},\dots,\gamma_r^{l_r}).$ 
		This is  non-zero, since the polynomials are assumed to be  linearly independent. 
		On noticing that $\gamma_1^{s_1}\gamma_2^{s_2}\cdots\gamma_r^{s_r}$ must divide $W$, so it follows that 
		$$
		\deg W\geq \sum_{1\leq i\leq r}s_id_i.
		$$ 
		Let $e=\max_{1\leq i\leq r}e_i$ and assume without loss of generality that $e=e_1$.
		It follows from	\eqref{eq:gamman} that the Wronskian $W(b,\gamma_2^{l_2},\dots,\gamma_r^{l_r})$ is equal to $W$. Hence
		\begin{align*}
			\deg W&\leq e_2+e_3+\dots+e_r-1-2-\dots-(r-1)\\ &=\left(\sum_{1\leq i\leq r}e_i\right)-e-\frac{r(r-1)}{2}.
		\end{align*}
		But then 
		$$
		e\leq \sum_{1\leq i\leq r}e_i-\sum_{1\leq i\leq r}s_id_i-\frac{r(r-1)}{2}=\sum_{1\leq i\leq r} \left(d_i\left(l_i-s_i\right)-\frac{r-1}{2}\right).
		$$
		Clearly 
		$l_i-s_i
		\leq l_i-(l_i-r+1)=r-1
		$, which leads to the statement of the lemma.
	\end{proof}

	\begin{corollary}\label{cor:coop}
		Assume that  $k\geq 13$ is odd,  $k> l> m\geq 2$,
		and $N\in \bar\QQ^*$.  Then any lines or conics 
		contained in the hypersurface 
		$ 		x_1^{k}+x_2^{l}+x_3^{m}+x_4^{k}=N$ must be contained in one of the subvarieties cut out by the union of equations 
		\begin{equation}\label{eq:tesco}
			V(x_1^k+x_4^k)\cup V(x_2^{l}+x_3^{m})\cup V(x_1^k+x_2^{l}+x_4^k)
			\cup V(x_1^k+x_3^{m}+x_4^k).
		\end{equation}
	\end{corollary}
	\begin{proof}
		Assume that there is a  line or conic contained in the hypersurface, so that 
		\begin{equation}\label{eq:gamma4n}
			\gamma_1(t)^{k}+\gamma_2(t)^{l}+\gamma_3(t)^{m}+\gamma_4(t)^{k}=N,
		\end{equation}
		with $\gamma_i\in \bar\QQ[t]$ and 
		$d_i=\deg(\gamma_i)\leq 2$, for $1\leq i\leq 4$, which are not all constant. 
		We begin by dispatching the case that one of the polynomials $\gamma_1,\dots,\gamma_4$ is constant. 
		
		We first assume that $d_4=0$.
		If $d_1=2$, then the
		left hand side  of  \eqref{eq:gamma4n} has degree $2k$. If $d_1=1$ and $d_2=2$, then the
		left hand side of \eqref{eq:gamma4n} has degree $\max\{k,2l\}$. 
		(Notice that $2l\neq k$, since $k$ is odd.) 
		If $d_1=d_2=1$ and $d_3=2$, then the degree is $\max\{k,2m\}$. If $d_1=1$ and $d_2,d_3\leq 1$, then the degree is $k$. In all of these cases the 
		left hand side of \eqref{eq:gamma4n} has strictly  positive degree and so cannot be equal to the constant $N$. Hence  $\gamma_1$ must also be constant. If $\gamma_1^k+\gamma_4^k=N$, then the curve lies on the surface cut out by the equation $x_2^{l}+x_3^{m}=0$ and we are done. If $\gamma_1^k+\gamma_4^k\neq N$, then $\gamma_2(t)^{l}+\gamma_3(t)^{m}=b$ for $b\in \bar\QQ^*$. Moreover,  $\gamma_2^{l}$ and $\gamma_3^{m}$ must be linearly independent, since at least one of  $\gamma_2$ or $\gamma_3$ is non-constant. But then it follows from Lemma \ref{lemma:wron} that
		$$
		\max\{d_2l,d_3 m\}\leq d_2+d_3-1.
		$$
		If $d_2=2$, then $2l\leq 1+d_3\leq 3$, which is impossible. If $d_2\leq1$, then $d_3m\leq 1+d_3-1$, which is also impossible. Thus, the case $d_4=0$ is completed. Notice that the case $d_1=0$ is identical, so from now on we can assume $d_1d_4\neq 0$.
		
		We now assume $d_3=0$ (and that $d_1d_4\neq 0$). If $\gamma_3(t)^m=N$, then the curve lies on $x_1^k+x_2^{l}+x_4^k=0$ and we are done. If not, we have $\gamma_1(t)^{k}+\gamma_2(t)^{l}+\gamma_4(t)^{k}=N'$ with $N'\neq 0$. If the three polynomials are linearly independent, then Lemma \ref{lemma:wron} implies that
		$$
		k\leq \max\{kd_1,ld_2,kd_4\}\leq 2\cdot 6-3=9,
		$$
		which is impossible. If they are linearly dependent, we must have $a_1\gamma_1(t)^{k}+a_2\gamma_2(t)^{l}=N'$ for $a_1,a_2\in\bar\QQ$ or $a_1\gamma_1(t)^{k}+a_4\gamma_4(t)^{k}=N'$ for $a_1,a_2\in\bar\QQ$. In the first case, if $a_1=a_2=0$, then the equation cannot hold. If $a_1\neq 0$, then $\deg(a_1\gamma_1(t)^{k})\neq \deg(a_2\gamma_2(t)^{l})$ and the equation cannot hold. If $a_1=0$ and $a_2\neq 0$, then $d_2=0$ and we would have $\gamma_1(t)^{k}+\gamma_4(t)^{k}=N''$ with $N''\in\bar\QQ$. If $N''=0$, then the curve lies on $V(x_1^k+x_4^k)$. If not, the equation cannot hold, on applying  Lemma \ref{lemma:wron} once again. The case $a_1\gamma_1(t)^{k}+a_4\gamma_4(t)^{k}=N'$ is similar and so the case $d_3=0$ is done. Notice that the case $d_2=0$ is identical, and so the case $d_i=0$ for some $i$ is now completed.
		
		We  may now proceed under the assumption that $d_i\in  \{1,2\}$, for all $1\leq i\leq 4$.
		For notational convenience we write $(l_1,l_2,l_3,l_4)=(k,l,m,k)$.
		Let $r\in \NN$ be the least  integer such that we can write
		$$			\gamma_1(t)^{l_1}+\gamma_2(t)^{l_2}+\gamma_3(t)^{l_3}+\gamma_4(t)^{l_4}=
		\sum_{j=1}^r a_{i_j}\gamma_{i_j}(t)^{l_{i_j}},
		$$
		for  $\{i_1,\dots, i_r\}\subseteq \{1,2,3,4\}$, with $a_{i_j}\in \bar\QQ^*$ and   $\gamma_{i_1}^{l_{i_1}},\dots ,\gamma_{i_r}^{l_{i_r}}$  linearly independent.
		
		If $r=4$ then  Lemma \ref{lemma:wron} implies that 
		$$
		\max\{d_1,d_4\}k\leq  \max_{1\leq j\leq 4}d_jl_j
		\leq 
		3(d_1+\dots+d_4)-6.
		$$
		If $d_1=d_4=1$ then this  implies that 
		$k\leq 18-6=12$, which is a contradiction. 
		If $\max\{d_1,d_4\}=2$ then we obtain 
		$2k\leq 24-6=18$, which is also impossible.
		If $r=3$ then at least one of $\gamma_1$ or $\gamma_4$ must belong to the sum.
		Assuming without loss of generality we fall in the former case, 
		it follows from
		Lemma \ref{lemma:wron} that 
		$$
		d_1k\leq  \max_{1\leq j\leq 3}d_{i_j}l_{i_j}
		\leq 2(d_{1}+d_{i_2}+d_{i_3})-3.
		$$
		If $d_1=1$ this implies that $k\leq 10-3=7$,
		whereas if $d_1=2$ then we obtain 
		$2k\leq 12-3=9$. 	
		Neither of these is possible.  
		Suppose next that $r=2$. Then 
		Lemma \ref{lemma:wron} implies that
		$$
		\max\{d_{i_1}l_{i_1},d_{i_2}l_{i_2}\}\leq d_{i_1}+d_{i_2}-1.
		$$
		But $\min\{l_{i_1},l_{i_2}\}\geq 2$, whence 
		$$
		2\max\{d_{i_1},d_{i_2}\}\leq \max\{d_{i_1}l_{i_1},d_{i_2}l_{i_2}\}\leq d_{i_1}+d_{i_2}-1\leq 2\max\{d_{i_1},d_{i_2}\}-1,
		$$
		which  is impossible.
		Finally,  the case $r=1$ is impossible since  we would then have $a_i\gamma_{i}(t)^{l_i}=N$, for a non-constant polynomial $\gamma_i$, for some $i\in \{1,\dots,4\}$.  
	\end{proof}

	We may now conclude the proof of Corollary \ref{thm:kl}, for which it remains to control the contribution from 
	lines or conics contained in the hypersurface $X\subset \mathbb{A}^4$ defined by
	$$
	x_1^k+x_2^{l}+x_3^{m}+x_4^k=N.
	$$ 
	Assuming  that  $k,l,m,N$ satisfy the hypotheses of the corollary, it follows from Corollary
	\ref{cor:coop} that lines or conics on the hypersurface are all contained in the union of surfaces cut out by \eqref{eq:tesco}.  Let $Y\subset \mathbb{A}^3$ be the surface obtained by intersecting 
	$X$ with the first  equation $x_1^k+x_4^k=0$ in \eqref{eq:tesco}. 
	There are $O(B)$ choices with  $|x_1|,|x_4|\leq B$ and $x_1^k+x_4^k=0$. We are then left with counting $|x_2|,|x_3|\leq B$ for which 
	$x_2^{l}+x_3^{m}=N$. 
	It follows  from
	Eisenstein's criterion over $\bar{\QQ}[x_3]$, as previously, 
	that	the underlying polynomial is  absolutely irreducible of degree $l$.
	But then
	Bombieri--Pila \cite[Thm.~5]{BP} reveals there to be $O_\ve(B^{1/l+\ve})$ choices for $x_2,x_3$. 
	In summary, the integral points 
	$\uu{x}\in S(\ZZ)$  make a satisfactory  overall contribution $O_\ve(B^{4/3+\ve})$ to 
	$N_2(B)$, since $l\geq 3$.
	The treatment of the other subvarieties defined by \eqref{eq:tesco} is similar, which thereby completes the proof of Corollary \ref{thm:kl}.


	
	\section*{Acknowledgments} 
The authors are grateful to Jakob Glas  
for helpful conversations and to 
Per Salberger for comments that led to an improvement in Corollary~\ref{thm:squares}.
Thanks are also due to the anonymous referee for numerous useful comments. 


	\bibliographystyle{amsplain}

%
%
	
	\begin{dajauthors}
		\begin{authorinfo}[tb]
			Tim Browning\\
			Institute of Science and Technology Austria, Am Campus 1, 3400 Klosterneuburg, Austria\\
			tdb\imageat{}ist.ac.at
		\end{authorinfo}
		\begin{authorinfo}[mv]
	Matteo Verzobio\\
	Institute of Science and Technology Austria, Am Campus 1, 3400 Klosterneuburg, Austria\\
	matteo.verzobio\imageat{}gmail.com
		\end{authorinfo}
	\end{dajauthors}
	
\end{document}